\documentclass[11pt]{amsart}

\usepackage{amssymb,amsmath,amscd}
\usepackage{hyperref,MnSymbol,enumitem}

\newcommand{\A}{\mathbb{A}}
\newcommand{\Q}{\mathbf{Q}}

\renewcommand{\P}{\mathfrak{P}}
\newcommand{\p}{\mathfrak{p}}
\newcommand{\q}{\mathfrak{q}}

\newcommand{\calE}{\mathcal E}
\newcommand{\calF}{\mathcal F}
\newcommand{\calO}{\mathcal O}
\newcommand{\calX}{\mathcal X}

\renewcommand{\to}{\rightarrow}

\newcommand{\disc}{\operatorname{disc}}

\newcommand{\Gal}{\operatorname{Gal}}

\newcommand{\Spec}{\operatorname{Spec}}

\newtheorem{thm}{Theorem}[section]
\newtheorem{lem}[thm]{Lemma}
\newtheorem{prop}[thm]{Proposition}
\newtheorem{cor}[thm]{Corollary}

\theoremstyle{definition}

\newtheorem{alg}[thm]{Algorithm}

\numberwithin{equation}{section}

\title[explicit hilbert irreducibility]{explicit hilbert irreducibility}

\author{David Krumm}
\address{Department of Mathematics and Statistics\\
Colby College}
\email{dkrumm@colby.edu}
\urladdr{http://personal.colby.edu/\textasciitilde dkrumm/}

\begin{document}

\maketitle

\begin{abstract}
Let $P(T,X)$ be an irreducible polynomial in two variables with rational coefficients. It follows from Hilbert's Irreducibility Theorem that for most rational numbers $t$ the specialized polynomial $P(t,X)$ is irreducible and has the same Galois group as $P$. We discuss here a method for obtaining an explicit description of the set of exceptional numbers $t$, i.e., those for which $P(t,X)$ is either reducible or has a different Galois group than $P$. To illustrate the method we determine the exceptional specializations of two polynomials of degrees four and six.
\end{abstract}

\section{Introduction}

Let $P\in \Q[T,X]$ be an irreducible polynomial in two variables with rational coefficients. Regarding $P$ as an element of the ring $\Q(T)[X]$, let $G$ be the Galois group of $P$, i.e., the Galois group of a splitting field for $P$ over $\Q(T)$. For any rational number $t$ we may consider the specialized polynomial $P_t=P(t,X)\in\Q[X]$ and its Galois group, which we denote by $G_t$. The Hilbert Irreducibility Theorem (henceforth abbreviated HIT) implies that as $t$ varies over all rational numbers, most specializations $P_t$ remain irreducible and have Galois group isomorphic to $G$. However, there may exist rational numbers $t$ for which either $P_t$ is reducible or $G_t$ is not isomorphic to $G$; we will call the set of all such numbers the \textit{exceptional set} of $P$, denoted $\calE(P)$. The main purpose of this article is to develop a method for obtaining an explicit description of this exceptional set.

A standard step\footnote{See Lang \cite[Chap. 9, \S 1]{lang_diophantine} or Serre \cite[\S 3.3]{serre_topics}.} in the proof of HIT is to show that there exist a finite set $D\subset\Q$ and algebraic curves $C_1,\ldots, C_r$ having the following property: if $t\in\Q\setminus D$ is such that either $P_t$ is reducible or $G_t$ is not isomorphic to $G$, then $t$ is a coordinate of a rational point on one of the curves $C_i$ (or more generally, $t$ is in the image of a map $C_i\to\mathbf P^1$). Our approach to obtaining an explicit description of the set $\calE(P)$ is based on a method for finding such a set $D$ and curves $C_1,\ldots, C_r$. Theorem \ref{main_hit_thm_intro} below, which was motivated by Serre's treatment of HIT in \cite[\S 3.3]{serre_topics} and by results of D\`ebes-Walkowiak \cite[\S3.1]{debes-walkowiak}, allows us to reduce the problem of finding both a set $D$ and defining equations for curves $C_1,\ldots, C_r$ to problems in computational group theory and Galois theory. 

\begin{thm}\label{main_hit_thm_intro}
Let $\Delta(T)$ and $\ell(T)$ be the discriminant and leading coefficient of $P$, respectively. Let $M_1,\ldots, M_r$ be representatives of all the conjugacy classes of maximal subgroups of $G$. For $i=1,\ldots, r$, let $F_i$ be the fixed field of $M_i$ and let $f_i(T,X)$ be a monic irreducible polynomial in $\Q[T][X]$ such that $F_i/\Q(T)$ is generated by a root of $f_i(T,X)$. Suppose that $t\in \Q$ satisfies
\begin{equation}\label{intro_exceptional_set_eq}
\Delta(t)\cdot\ell(t)\cdot\prod_{i=1}^r\disc f_i(t,X)\ne 0.
\end{equation}
Then $t\in\calE(P)$ $\iff$ there is an index $i$ such that $f_i(t,X)$ has a root in $\Q$.
\end{thm}

It follows from Theorem \ref{main_hit_thm_intro} that we may take $D$ to be the finite set of rational numbers $t$ for which \eqref{intro_exceptional_set_eq} does not hold, and we may take $C_i$ to be the affine plane curve defined by the equation $f_i(T,X)=0$. Indeed, the theorem implies that -- disregarding elements of $D$ -- the set $\calE(P)$ consists of the first coordinates of all the rational points on the curves $C_i$. 

In practice this result can be used to explicitly describe the set $\calE(P)$ for any given polynomial $P$. All of the algebraic objects appearing in the theorem -- in particular the group $G$, the subgroups $M_i$, and the polynomials $f_i$ -- can be computed using currently available methods in computer algebra. Furthermore, depending on the geometry of the curves $C_i$, one may be able to determine the sets of rational points on all these curves, thus obtaining a complete characterization of the elements of $\calE(P)$.

A more general version of Theorem \ref{main_hit_thm_intro} is proved in \S\ref{hit_section}, and further details regarding the associated algorithmic questions are given in \S\ref{algorithm_section}. In order to illustrate the process described above, we include two examples in \S\ref{examples_section}.  The first example concerns the polynomial
\[P(T,X)=3X^4-4X^3+1+3T^2,\]
which is one polynomial in a family discussed by Serre \cite[\S4.5]{serre_topics}. The Galois group of $P$ is isomorphic to the alternating group $A_4$, so a typical specialization $P_t$ will have Galois group $G_t\cong A_4$. We show that there are infinitely many specializations of $P$ with Galois group different from $A_4$, and that these can parametrized. More precisely, we prove:
\[G_t\not\cong A_4\iff t=\frac{v^3 - 9v}{9(1 - v^2)}\;\;\text{for some}\;v\in\Q.\]
In the second example we consider the polynomial
\[P(T,X)=X^6+T^6-1.\]
The case $n=3$ of Fermat's Last Theorem implies that the only rational numbers $t$ for which $P_t$ has a rational root are 0 and $\pm 1$. We will prove the stronger result that in fact 0 and $\pm 1$ are the only rational numbers $t$ for which $P_t$ is reducible.

\subsection*{Acknowledgements} I thank Pierre D\`ebes for several helpful discussions related to the material of \S\ref{hit_section}.
\section{HIT via extensions of Dedekind domains}\label{hit_section}

Let $k$ be a field of characteristic 0 and let $P(T,X)\in k[T,X]$ be a polynomial of degree $n\ge 1$ in the variable $X$. We will henceforth regard $P$ as an element of the ring $k(T)[X]$ and assume that $P$ is separable. We define the \textit{factorization type} of $P$, denoted $\calF(P)$, to be the multiset consisting of the degrees of the irreducible factors of $P$. 

Let $N/k(T)$ be a splitting field of $P$ and let $G=\Gal(N/k(T))$ be the Galois group of $P$. We assume that $G$ is nontrivial. For every element $t\in k$, let $P_t$ denote the specialized polynomial $P(t,X)\in k[X]$. The Galois group and factorization type of $P_t$ will be denoted by $G_t$ and $\calF(P_t)$, respectively.

It follows from HIT that there is a thin\footnote{See \cite[\S 3.1]{serre_topics} for a definition and properties of thin sets.} subset of $k$ outside of which we have $\calF(P_t)=\calF(P)$ and $G_t\cong G$. We define the \textit{exceptional set} of $P$, denoted $\calE(P)$, to be the set of all elements $t\in k$ for which either one of these conditions fails to hold:
\[\calE(P)=\{t\in k\;\vert\;\calF(P_t)\ne\calF(P)\;\text{or}\;G_t\not\cong G\}.\]
Our aim in this section is to prove a version of HIT from which one can deduce a method for explicitly describing the set $\calE(P)$. Our main result in this direction is Theorem \ref{main_hit_thm} below.

Let $\Delta(T)$ and $\ell(T)$ be the discriminant and leading coefficient of $P$, respectively. Let $A\subset k(T)$ be the ring
\[A=k[T]\left[\ell(T)^{-1}\right].\]

For every intermediate field $F$ between $k(T)$ and $N$, let $\calO_F$ denote the integral closure of $A$ in $F$. Note that $\calO_F/A$ is an extension of Dedekind domains with $A$ being a PID. By a \textit{prime} of $F$ (or of $\calO_F)$ we mean a maximal ideal of $\calO_F$. If $\p$ is a prime of $A$ and $\q$ is a prime of $\calO_F$, we denote by $\kappa(\q)$ and $\kappa(\p)$ the residue fields of $\q$ and $\p$, respectively. Thus,
\[\kappa(\q)=\calO_F/\q,\;\;\kappa(\p)=A/\p.\]
If $\q$ divides $\p\calO_F$, we denote the ramification index and residual degree of $\q$ over $\p$ by $e(\q/\p)$ and $f(\q/\p)$, respectively.

For every prime $\P$ of $N$, let $G_{\P}$ be the decomposition group of $\P$ over $k(T)$ and let $Z_{\P}$ be the decomposition field of $\P$, i.e., the fixed field of $G_{\P}$. We refer the reader to \cite[Chap. I, \S\S 8-9]{neukirch} for the standard material on decomposition groups and ramification used in this section. 

If $t\in k$ is any element satisfying $\ell(t)\ne 0$, then the evaluation homomorphism $k[T]\to k$ given by $a(T)\mapsto a(t)$ extends uniquely to a homomorphism $A\to k$. Let $\p_t$ be the kernel of this map. We will henceforth identify the residue field $\kappa(\p_t)$ with $k$ via the map $a(T)\bmod\p_t\mapsto a(t)$. Note that with this identification, if $f(T,X)\in A[X]$ is an arbitrary polynomial, then upon reducing the coefficients of $f$ modulo $\p_t$ we obtain the specialized polynomial $f(t,X)\in k[X]$.

It will be necessary for our purposes in this section to be able to determine how the prime $\p_t$ factors in any intermediate field $F$ between $k(T)$ and $N$. Recall that by a theorem of Dedekind-Kummer, for all but finitely many primes $\p$ of $A$, the factorization of $\p$ in $F$ can be determined by choosing an integral primitive element $\theta$ of $F/k(T)$ and factoring its minimal polynomial modulo $\p$. The finite set of primes that need to be excluded are those that are not relatively prime to the conductor of the ring $A[\theta]$; see \cite[p. 47, Prop. 8.3]{neukirch} for details. The following lemma provides sufficient conditions on $t\in k$ so that $\p_t$ will be relatively prime to this conductor, and therefore the Dedekind-Kummer criterion can be applied to $\p_t$.

\begin{lem}\label{conductor_lem}
Let $F$ be an intermediate field between $k(T)$ and $N$ with primitive element $\theta\in\calO_F$ having minimal polynomial $f(T,X)\in A[X]$. Let
\[\mathfrak F=\{\alpha\in\calO_F\;\vert\;\alpha\cdot\calO_F\subseteq A[\theta]\}\]
be the conductor of the ring $A[\theta]$. Suppose that $t\in k$ satisfies
\[\ell(t)\cdot\disc f(t,X)\ne 0.\]
Then $\p_t\calO_F$ is relatively prime to $\mathfrak F$. Furthermore, $\p_t$ is unramified in $F$.
\end{lem}
\begin{proof}
Let $\delta\in A$ be the discriminant of $f$. By a linear algebra argument (see Lemma 2.9 in \cite[p. 12]{neukirch}) we have $\delta\cdot\calO_F\subseteq A[\theta]$ and therefore $\delta\in\mathfrak F$. Suppose that $\q$ is a prime of $F$ dividing both $\mathfrak F$ and $\p_t\calO_F$. Since $\mathfrak F\subseteq\q$ we have $\delta\in\q$, so $\delta\in\q\cap A=\p_t$. By definition of $\p_t$ this implies that $\disc f(t,X)=\delta(t)=0$, which is a contradiction. Therefore $\p_t$ must be relatively prime to $\mathfrak F$. 

The Dedekind-Kummer theorem now allows us to relate the factorization of $\p_t$ in $F$ to the factorization of $f(t,X)$ in $k[X]$. In particular, the theorem implies that if $\p_t$ is ramified in $F$, then $f(t,X)$ has a repeated irreducible factor, which contradicts our assumption that $\disc f(t,X)\ne 0$. Therefore $\p_t$ must be unramified in $F$.
\end{proof}

\begin{lem}\label{unramified_lem}
Suppose that  $t\in k$ satisfies $\Delta(t)\cdot\ell(t)\ne 0$. Then the prime $\p_t$ is unramified in $N$.
\end{lem}
\begin{proof}
Since $N$ is the compositum of the fields $k(T)(\theta)$ as $\theta$ ranges over the roots of $P$ in $N$, it suffices to show that $\p_t$ is unramified in every such field. (See \cite[p. 119, Cor. 8.7]{lorenzini}.) Thus, let $\theta\in \calO_N$ be any root of $P$ and let $F=k(T)(\theta)$. Let $Q\in k[T][X]$ be an irreducible factor of $P$ having $\theta$ as a root. Dividing $Q$ by its leading coefficient we obtain a monic irreducible polynomial $f\in A[X]$ having $\theta$ as a root; it follows that $f$ is the minimal polynomial of $\theta$ over $k(T)$. Let $\delta\in A$ be the discriminant of $f$. Since $f$ divides $P$ in $A[X]$, $\delta$ divides $\Delta$ in $A$. Hence, the hypothesis that $\Delta(t)\ne 0$ implies that $\delta(t)\ne 0$. The result now follows from Lemma \ref{conductor_lem}.
\end{proof}

\begin{prop}\label{decomposition_group_prop}
Suppose that $t\in k$ satisfies $\Delta(t)\cdot\ell(t)\ne 0$, and let $\P$ be a prime of $N$ dividing $\p_t$. Then $G_{\P}$ is isomorphic to $G_t$.
\end{prop}

\begin{proof}
For every element $a\in\calO_N$ let $\bar a$ denote the image of $a$ under the quotient map $\calO_N\to \kappa(\P)$. Recall that the extension $\kappa(\P)/k$ is Galois and that there is a surjective homomorphism $G_{\P}\to\Gal(\kappa(\P)/k)$ given by $\sigma\mapsto\bar\sigma$, where $\bar\sigma(\bar a)=\overline{\sigma(a)}$ for every $a\in\calO_N$. Furthermore, since $\p_t$ is unramified in $N$ by Lemma \ref{unramified_lem}, this map is an isomorphism. Hence, in order to prove the proposition it suffices to show that $\kappa(\P)$ is a splitting field for $P_t$.

Note that if $\alpha\in\calO_N$ is a root of $P$, then $\bar \alpha\in\kappa(\P)$ is a root of $P_t$. Moreover, if $\alpha$ and $\beta$ are distinct roots of $P$, then $\bar \alpha\ne\bar \beta$; indeed, this follows from the fact that $\bar\Delta=\Delta(t)\ne 0$. Thus, reduction modulo $\P$ is an injective map from the set of roots of $P$ to the set of roots of $P_t$.

Let $x_1,\ldots, x_n\in\calO_N$ be the roots of $P$ in $N$, and let $S=k(\bar x_1,\ldots, \bar x_n)$. Clearly $S$ is a splitting field for $P_t$, and $k\subseteq S\subseteq \kappa(\P)$. We will prove that $S=\kappa(\P)$ by showing that the group $\Gal(\kappa(\P)/S)$ is trivial. Let $\tau\in\Gal(\kappa(\P)/S)$ and let $\sigma\in G_{\P}$ be the element such that $\bar\sigma=\tau$. Since $\tau$ is the identity map on $S$, we have $\tau(\bar x_i)=\bar x_i$ for every index $i$, and hence $\overline{\sigma(x_i)}=\bar x_i$ for all $i$. Since $\sigma(x_i)$ and $x_i$ are roots of $P$, this implies that $\sigma(x_i)=x_i$. Thus, $\sigma$ fixes every root of $P$, so $\sigma$ is the identity element of $G_{\P}$. Hence $\tau=\bar\sigma$ is the identity element of $\Gal(\kappa(\P)/S)$. This proves that $\Gal(\kappa(\P)/S)$ is trivial and therefore $\kappa(\P)=S$ is a splitting field for $P_t$.
\end{proof}

\begin{lem}\label{decomposition_field_lem}
Let $\p$ be a prime of $A$ and let $\P$ be a prime of $N$ dividing $\p$. Then the following hold:
\begin{enumerate}
\item Setting $\mathfrak Q=\P\cap Z_{\P}$, we have $e(\mathfrak Q/\p)=f(\mathfrak Q/\p)=1$.
\item Let $F$ be an intermediate field between $k(T)$ and $N$, and let $\q=\P\cap F$. If $e(\q/\p)=f(\q/\p)=1$, then $F\subseteq Z_{\P}$.
\end{enumerate}
\end{lem}
\begin{proof}
See \cite[p. 55, Prop. 9.3]{neukirch} and \cite[p. 118, Prop. 8.6]{lorenzini}.
\end{proof}

\begin{prop}\label{decomposition_field_prop}
Let $F$ be an intermediate field between $k(T)$ and $N$. Let $\theta\in\calO_F$ be a primitive element for $F/k(T)$ and let $f(T,X)\in A[X]$ be the minimal polynomial of $\theta$. Suppose that $t\in k$ satisfies 
\[\Delta(t)\cdot\ell(t)\cdot\disc f(t,X)\ne 0.\]
Then the following are equivalent:
\begin{enumerate}
\item The polynomial $f(t,X)$ has a root in $k$.
\item There exists a prime $\q$ of $F$ dividing $\p_t$ such that $f(\q/\p_t)=1$.
\item There exists a prime $\P$ of $N$ dividing $\p_t$ such that $F\subseteq Z_{\P}$.
\end{enumerate}
\end{prop}
\begin{proof}
By Lemma \ref{conductor_lem}, $\p_t$ is relatively prime to the conductor of $A[\theta]$. The Dedekind-Kummer theorem then implies that the degrees of the irreducible factors of $f(t,X)$ in $k[X]$ correspond to the residual degrees $f(\q/\p_t)$ for primes $\q$ of $F$ dividing $\p_t$. The equivalence of (1) and (2) follows immediately.

We now show that (2) and (3) are equivalent. Suppose that (2) holds, and let $\P$ be a prime of $N$ dividing $\q$. By Lemma \ref{unramified_lem}, $\p_t$ is unramified in $N$ and therefore unramified in $F$. Hence, $e(\q/\p_t)=1$. By Lemma \ref{decomposition_field_lem}, $F\subseteq Z_{\P}$. Thus, (3) holds.

Finally, suppose that (3) holds. Let $\P$ be a prime of $N$ dividing $\p_t$ such that $F\subseteq Z_{\P}$. Let $\mathfrak Q=\P\cap Z_{\P}$ and $\q=\P\cap F$. Since $f(\mathfrak Q/\p_t)=1$ and $f(\q/\p_t)$ divides $f(\mathfrak Q/\p_t)$, we have $f(\q/\p_t)=1$. Thus, (2) holds.
\end{proof}

\begin{thm}\label{main_hit_thm}
Let $M_1,\ldots, M_r$ be representatives of all the conjugacy classes of maximal subgroups of $G$. For $i=1,\ldots, r$ let $F_i$ be the fixed field of $M_i$, and let $f_i(T,X)$ be a monic irreducible polynomial in $k[T][X]$ such that $F_i/k(T)$ is generated by a root of $f_i(T,X)$. Suppose that $t\in k$ satisfies
\[\Delta(t)\cdot\ell(t)\cdot\prod_{i=1}^r\disc f_i(t,X)\ne 0.\]
Then the following hold:
\begin{enumerate}
\item If  $\calF(P_t)\ne\calF(P)$, then $G_t\not\cong G$.
\item $G_t\not\cong G$ $\iff$ there is an index $i$ such that $f_i(t,X)$ has a root in $k$.
\end{enumerate}
\end{thm}

\begin{proof}
We begin by proving (1). Thus, suppose that $\calF(P_t)\ne\calF(P)$. Let $P_1,\ldots, P_s\in A[X]$ be monic irreducible polynomials such that
\[P=\ell(T)\cdot P_1\cdots P_s.\]
Since $\calF(P_t)\ne\calF(P)$, there exists $f\in\{P_1,\ldots, P_s\}$ such that $f(t,X)$ is reducible. Let $\theta\in \calO_N$ be a root of $f$ and let $F=k(T)(\theta)$. Since $\disc f$ divides $\disc P=\Delta$ and $\Delta(t)\ne 0$, then $\disc f(t,X)\ne 0$. Lemma \ref{conductor_lem} implies that $\p_t\calO_F$ is relatively prime to the conductor of $A[\theta]$; we may therefore apply the Dedekind-Kummer theorem to relate the factorization of $f(t,X)$ to the factorization of $\p_t\calO_F$.

Since $f(t,X)$ is separable and reducible, it must have more than one irreducible factor (up to associates). Hence, there is more than one prime of $F$ dividing $\p_t$, and therefore more than one prime of $N$ dividing $\p_t$. It follows that if $\P$ is any prime of $N$ dividing $\p_t$, the group $G_{\P}$ is a proper subgroup of $G$. (Indeed, the index $|G:G_{\P}|$ is the number of primes of $N$ dividing $\p_t$.) Proposition \ref{decomposition_group_prop} now implies that $G_t\not\cong G$, which proves (1).

We now prove (2). Suppose that $G_t\not\cong G$ and let $\P$ be a prime of $N$ dividing $\p_t$. By Proposition \ref{decomposition_group_prop}, the group $G_{\P}$ is a proper subgroup of $G$. Replacing $\P$ by a conjugate ideal if necessary, we may therefore assume that $G_{\P}\subseteq M_i$ for some index $i$. The decomposition field $Z_{\P}$ then contains $F_i$, and by Proposition \ref{decomposition_field_prop} applied to the field $F_i$, this implies that $f_i(t,X)$ has a root in $k$. This proves one direction of (2). The converse follows by a similar argument.
\end{proof}

It follows from the above theorem that the problem of determining the exceptional set of $P$ can be reduced to a problem of finding all the $k$-rational points on a finite list of curves. More precisely, we have the following.

\begin{cor}\label{exceptional_set_cor}
With notation as in Theorem \ref{main_hit_thm}, let $D$ be the finite set of all elements $t\in k$ such that
\[\Delta(t)\cdot\ell(t)\cdot\prod_{i=1}^r\disc f_i(t,X)= 0.\]
For $i=1,\ldots, r$ let $C_i$ be the affine plane curve defined by the equation $f_i(T,X)=0$. Let $t\in k\setminus D$. Then
$t\in\calE(P)$ if and only if $t$ is the first coordinate of a $k$-rational point on one of the curves $C_i$.
\end{cor}

\section{Algorithmic aspects}\label{algorithm_section}

Theorem \ref{main_hit_thm} suggests the following algorithm which can be used to study the exceptional set of the polynomial $P$. We state the algorithm first and then explain its precise relation to this problem.

\begin{alg}\label{hit_data_algorithm}\mbox{}

\smallskip
\noindent \textit{Input:} \hspace{1mm} A separable polynomial $P\in k[T][X]$.\\
\noindent\textit{Output:} A finite set $D\subset k$ and a finite set $S\subset k[T][X]$.

\begin{enumerate}
\item Create empty sets $D$ and $S$.
\item Include in $D$ all the $k$-roots of the leading coefficient of $P$.
\item Include in $D$ all the $k$-roots of the discriminant of $P$.
\item Compute the group $G=\Gal(P)$. More precisely, find a permutation representation of $G$ induced by a labeling of the roots of $P$.
\item Find subgroups $M_1,\ldots, M_r$ representing all the conjugacy classes of maximal subgroups of $G$.
\item For $M\in\{M_1,\ldots, M_r\}$: 
\begin{enumerate}
\item Find a monic irreducible polynomial $f\in k[T][X]$ such that the fixed field of $M$ is generated by a root of $f$.
\item Include $f$ in the set $S$.
\item Include in $D$ all the $k$-roots of the discriminant of $f$.
\end{enumerate}
\item Return the sets $D$ and $S$.
\end{enumerate}
\end{alg}

\begin{thm}\label{algorithm_output_thm}
Let $P\in k[T][X]$ be a separable polynomial, and let $D$ and $S$ form the output of Algorithm \ref{hit_data_algorithm} with input $P$. Then the following hold for all $t\in k\setminus D$:
\begin{enumerate}
\item If  $\calF(P_t)\ne\calF(P)$, then $G_t\not\cong G$.
\item $G_t\not\cong G$ $\iff$ there exists $f\in S$ such that $f(t,X)$ has a root in $k$.
\end{enumerate}
\end{thm}
\begin{proof}
This is an immediate consequence of Theorem \ref{main_hit_thm}.
\end{proof}

In the case $k=\Q$, all of the computational methods needed to carry out the steps of Algorithm \ref{hit_data_algorithm} are known, and most have been implemented in computer algebra systems. Indeed:

\begin{itemize}[leftmargin=6mm]
\item A permutation representation of $G$ can be computed by using an algorithm of Fieker-Kl\"uners \cite{fieker-kluners}.
\item A set of representatives for the conjugacy classes of maximal subgroups of $G$ can be obtained using an algorithm of Cannon-Holt \cite{cannon-holt}.
\item Given a subgroup $H\le G$, the minimal polynomial of a primitive element of the fixed field of $H$ can be found using a method discussed in \cite[\S 3.3]{kluners-malle}.
\end{itemize}

Most of the above algorithms have been implemented and are included in \textsc{Magma} \cite{magma}; the only exception is the computation of Galois groups of reducible polynomials over $\Q(T)$.  Hence, there is at present an obstacle to carrying out Algorithm \ref{hit_data_algorithm} with a reducible polynomial as input. However, this problem is being addressed in current work of Nicole Sutherland, and an implementation of the algorithm of Fieker-Kl\"uners for reducible polynomials will be included in a future release of \textsc{Magma}.

In view of the above discussion, it is currently possible to translate the problem of determining the exceptional set of an irreducible polynomial $P\in\Q[T,X]$ to a problem of determining the sets of rational points on a finite list of algebraic curves. The difficulty of the problem is therefore largely dependent on the genera of these curves; if the genera are not too large, it may be possible to obtain an explicit characterization of the set $\calE(P)$. For a survey of the presently available methods for computing rational points on curves, we refer the reader to Stoll's article \cite{stoll_survey}.

\section{Examples}\label{examples_section}
Having developed the theoretical and algorithmic material that form the core of this article, we proceed to apply our results to study the exceptional sets of two sample polynomials, one with an infinite exceptional set and one with a finite exceptional set. In order to carry out the necessary computations, an implementation of Algorithm \ref{hit_data_algorithm} in \textsc{Magma} will be used. The source code of our implementation is available in \cite{code}.

We include a cautionary remark for the reader who may be interested in reproducing our calculations. The method used by \textsc{Magma} to find primitive elements of fixed fields (which is needed in step 6(a) of Algorithm \ref{hit_data_algorithm}) does not always produce the same primitive element for a given field extension. Hence, the output of Algorithm \ref{hit_data_algorithm} that the reader obtains may be different from what is given here. However, in that case the arguments made below can be easily adapted to prove the same results.

\subsection{An infinite exceptional set}\label{serre_section} In \cite[\S4.5]{serre_topics} Serre shows that for even values of $n$, the polynomial 
\[P_n(T,X)=(n-1)X^n - nX^{n-1}+1+(-1)^{n/2}(n-1)T^2\]
has the alternating group $A_n$ as its Galois group. By HIT, most specializations $P_n(t,X)$ will have Galois group $A_n$ as well. In the case $n=4$ we obtain the polynomial
\[P(T,X)=3X^4-4X^3+1+3T^2\]
with Galois group $A_4$. We will now determine precisely which specializations of $P$ have Galois group different from $A_4$.

\begin{lem}\label{serre_ex_F1_lem}
Let $F_1(T,X) = X^4 + 4X^3 + 81T^2 + 27$ and let $t\in\Q^{\ast}$. Then the polynomial $F_1(t,X)$ has no rational root.
\end{lem}
\begin{proof}
Suppose that there exists $x\in\Q$ such that $F_1(t,x)=0$. Since $t\ne 0$, we must have $x\ne-3$. Defining $y=9t/(x+3)$, the equation $F_1(t,x)=0$ implies that
\[y^2=-(x^2-2x+3).\]
However, a simple argument\footnote{The solvability of an equation of the form $y^2=f(x)$ over any given $p$-adic field can be tested using a method of Bruin \cite[\S 5.4]{bruin_local_solvability} which is implemented in the \textsc{Magma} function \texttt{HasPoint}.} shows that the above equation has no solution in $\Q_2$ and therefore no solution in $\Q$. This contradiction proves the lemma.
\end{proof}

\begin{lem}\label{serre_ex_F2_lem}
Let $F_2(T,X) = X^3 + 48X^2 + (336-1296T^2)X - 10368T^2 + 640$ and let $t\in\Q^{\ast}$. Then the polynomial $F_2(t,X)$ has a rational root if and only if $t$ has the form
\begin{equation}\label{serre_ex_eq}
t=\frac{v^3 - 9v}{9(1 - v^2)}
\end{equation}
for some rational number $v$.
\end{lem}
\begin{proof}
Let $C$ be the plane curve defined by the equation $F_2(T,X)=0$. The curve $C$ is parametrizable; indeed, the rational maps $\phi:C\dashedrightarrow\A^1=\Spec\Q[V]$ and $\psi:\A^1\dashedrightarrow C$ given by
\[\psi(V)=\left(\frac{V^3-9V}{9(1-V^2)},\frac{8(V^2-5)}{1-V^2}\right)\;\;\text{and}\;\;\phi(T,X)=\frac{X^2 - 1296T^2 + 44X + 160}{144T}\]
are easily seen to be inverses.

Suppose that $t$ is of the form \eqref{serre_ex_eq}. We may then define
\[x=\frac{8(v^2-5)}{1-v^2},\]
so that $\psi(v)=(t,x)$ is a rational point on $C$. Hence, the polynomial $F_2(t,X)$ has a rational root (namely $x$).

Conversely, suppose that $F_2(t,X)$ has a rational root, say $x$. Since $t\ne 0$, the map $\phi$ is defined at the point $(t,x)\in C(\Q)$. Thus, we may define $v=\phi(t,x)$. We claim that $v\ne\pm 1$. A straightforward calculation shows that the rational points on the pullback of $\pm 1$ under $\phi$ are $(0,-40)$ and $(0,-4)$. Since $t\ne 0$, the point $(t,x)$ is different from these two points. Hence $v=\phi(t,x)\ne\pm 1$, as claimed. The map $\psi$ is therefore defined at $v$, so $(t,x)=\psi(v)$. In particular, $t$ is of the form \eqref{serre_ex_eq}.
\end{proof}

\begin{prop}\label{serre_specializations_prop}
Let $t\in\Q$ and let $G_t$ be the Galois group of $P_t$. Then
\[G_t\not\cong A_4\iff t=\frac{v^3 - 9v}{9(1 - v^2)}\;\;\text{for some}\;v\in\Q.\]
\end{prop}
\begin{proof}
For $t=0$ the proposition holds because both statements in the above equivalence are true. Indeed, we have
\[P_0=3X^4-4X^3+1=(X-1)^2(3X^2+2X+1),\]
so $G_0$ has order 2. Suppose now that $t\ne 0$. 

Applying Algorithm \ref{hit_data_algorithm} to the polynomial $P$ we obtain the set $\{0\}$ and the polynomials
\begin{align*}
F_1(T,X) &= X^4 + 4X^3 + 81T^2 + 27,\\
F_2(T,X) &= X^3 + 48X^2 + (-1296T^2 + 336)X - 10368T^2 + 640.
\end{align*}

By Theorem \ref{algorithm_output_thm} and Lemmas \ref{serre_ex_F1_lem} and \ref{serre_ex_F2_lem}, we have the following:
\begin{align*}
G_t\not\cong A_4 &\iff F_1(t,X)\cdot F_2(t,X)\;\text{has a rational root}\\
&\iff F_2(t,X)\;\text{has a rational root}\\
&\iff t=\frac{v^3 - 9v}{9(1 - v^2)}\;\;\text{for some}\;v\in\Q.
\end{align*}
This completes the proof.
\end{proof}

\subsection{A finite exceptional set}\label{fermat_section} In our second example we consider the polynomial \[P(T,X)=X^6+T^6-1.\]
As follows from the case $n=3$ of Fermat's Last Theorem, the specialized polynomial $P_t$ has a rational root if and only if $t\in\{0,\pm 1\}$. We will now prove the following stronger result.

\begin{prop}\label{fermat_specializations_prop}
For $t\in\Q$, the polynomial $P_t$ is reducible if and only if $t\in\{0,\pm1\}$.
\end{prop}
\begin{proof}
Suppose that $P_t$ is reducible. We will show by contradiction that $t\in\{0,\pm1\}$. Thus, suppose that $t\notin\{0,\pm1\}$.

Applying Algorithm \ref{hit_data_algorithm} to the polynomial $P$ we obtain the set $\{-1,1\}$ and the polynomials
\begin{align*}
F_1(T,X) &= X^2 - 2^8\cdot 3^5\left((T - 1)(T + 1)(T^2 - T + 1)(T^2 + T + 1)\right)^3,\\
F_2(T,X) &= X^2 + 64\cdot 27\left((T - 1)(T + 1)(T^2 - T + 1)(T^2 + T + 1)\right)^2,\\
F_3(T,X) &= X^2 + 12X + 27 + 9T^6,\\
F_4(T,X) &= X^3 + 12X^2 + 48X + 72 - 8T^6.
\end{align*}

By Theorem \ref{algorithm_output_thm}, one of the polynomials $F_i(t,X)$ must have a rational root; we accordingly divide the proof into four cases.

\underline{\textit{Case 1:}} There exists $x\in\Q$ such that $F_1(t,x)=0$. Defining
\[v=x/\left(2^4\cdot 3^2\cdot(t - 1)(t + 1)(t^2 - t + 1)(t^2 + t + 1)\right),\]
the equation $F_1(t,x)=0$ implies that
\[v^2=3(t - 1)(t + 1)(t^2 - t + 1)(t^2 + t + 1).\]

The above equation defines a hyperelliptic curve $\calX$ of genus 2. By a descent argument one can show that the Jacobian variety of $\calX$ has a Mordell-Weil group of rank 0; it is therefore a straightforward calculation\footnote{Stoll's algorithm of 2-descent \cite{stoll_descent} is implemented in \textsc{Magma} and can be accessed via the \texttt{RankBound} function. Once the rank of the Jacobian is known to be 0, the \texttt{Chabauty0} function carries out the calculation of finding all the rational points on $\calX$.} to determine the set of rational points on $\calX$. We find that the only rational points are the Weierstrass points, namely $(\pm 1,0)$. It follows that $t=\pm 1$, which is a contradiction.

\underline{\textit{Case 2:}} There exists $x\in\Q$ such that $F_2(t,x)=0$. Letting
\[u=8\cdot 3\cdot(t - 1)(t + 1)(t^2 - t + 1)(t^2 + t + 1),\]
we have $u\ne 0$ and $x^2+3u^2=0$, which is clearly impossible. Thus we have a contradiction.

\underline{\textit{Case 3:}} There exists $x\in\Q$ such that $F_3(t,x)=0$. Defining
\[v=\frac{x+6}{3}\;\;\text{and}\;\; u=-t^2,\]
the equation $F_3(t,x)=0$ implies that
\[v^2=u^3+1.\]
The above equation defines the elliptic curve with Cremona label 36a1. This curve has rank 0 and a torsion subgroup of order 6; its only affine rational points are
\[(0,\pm 1), (2,\pm 3),\;\text{and}\;(-1,0).\] 
It follows from this that $u=0$, 2, or $-1$. This implies, respectively, that $t=0$, $t^2=-2$, or $t^2=1$, all of which lead to a contradiction. 

\underline{\textit{Case 4:}} There exists $x\in\Q$ such that $F_4(t,x)=0$. Letting $y=4t^3$, the equation $F_4(t,x)=0$ implies that
\[y^2=2(x^3+12x^2+48x+72).\]
The above equation defines the elliptic curve with Cremona label 36a1, the same curve that appeared in the previous case. Using the above model of the curve, the affine rational points are
\[(0,\pm 12), (-4,\pm 4),\; \text{and}\; (-6,0).\]
It follows that $y=\pm 12, \pm 4$, or 0, which implies that $t^3=\pm 3$, $t=\pm 1$, or $t=0$, all of which yield a contradiction.

Since every case has led to a contradiction, we conclude that $t\in\{0,\pm 1\}$. This completes the proof of  the proposition.
\end{proof}

\bibliography{ref_list}
\bibliographystyle{amsplain}
\end{document}